\documentclass[a4paper,11pt]{amsart}
\usepackage{amsmath,amssymb}

\usepackage{amsthm}

\newcommand{\bigpare}[1]{\bigl(#1\bigr)}
\newcommand{\biggpare}[1]{\biggl(#1\biggr)}
\newcommand{\Bigpare}[1]{\Bigl(#1\Bigr)}

\newcommand{\biggbra}[1]{\biggl\{#1\biggr\}}

\newcommand{\Bigbrac}[1]{\Bigl[#1\Bigr]}

\newcommand{\bigset}[2]{\bigl\{#1\bigm|#2\bigr\}}

\newcommand{\biggset}[2]{\biggl\{#1\biggm|#2\biggr\}}

\newcommand{\norm}[1]{\| #1 \|}
\newcommand{\bignorm}[1]{\bigl\| #1 \bigr\|}

\newcommand{\bigabs}[1]{\bigl| #1 \bigr|}

\newcommand{\jap}[1]{\langle #1 \rangle}

\def\a{\alpha}
\def\b{\beta}
\def\c{\gamma}
\def\d{\delta}
\def\e{\varepsilon}
\def\f{\varphi}
\def\g{\psi}

\def\k{\kappa}
\def\l{\lambda}
\def\s{\sigma}

\def\x{\xi}
\def\y{\eta}
\def\z{\zeta}

\newcommand{\F}{\Phi}
\newcommand{\G}{\Psi}

\renewcommand{\O}{\Omega}
\renewcommand{\S}{\Sigma}

\newcommand{\Op}{\mathrm{Op}}

\def\re{\mathbb{R}}
\def\co{\mathbb{C}}
\def\ze{\mathbb{Z}}

\def\pa{\partial}

\newcommand{\supp}{\mathrm{{supp}}}

\newcommand{\WF}{\mathrm{WF}}

\newtheorem{thm}{Theorem}[section]
\newtheorem{lem}[thm]{Lemma}
\newtheorem{prop}[thm]{Proposition}
\newtheorem{cor}[thm]{Corollary}

\theoremstyle{definition}

\theoremstyle{remark}
\newtheorem{rem}{Remark}[section]

\title{Microlocal resolvent estimates, revisited}
\author{Shu Nakamura}
\address{Graduate School of Mathematical Sciences, University of Tokyo,
 3-8-1, Komaba, Meguro-ku, Tokyo, Japan 153-8914. 
Email: {\tt shu@ms.u-tokyo.ac.jp}.}

\begin{document}
\label{firstpage}
\maketitle

\begin{abstract}
Let $H$ be a Schr\"odinger type operator with long-range perturbation. 
We study the wave front set of the distribution kernel of $(H-\l\mp i0)^{-1}$, 
where $\l$ is in the absolutely continous spectrumof $H$.
The result is a refinement of the microlocal resolvent 
estimate of Isozaki-Kitada \cite{IK1,IK2}. We prove the result for a class of 
pseudodifferential operators on manifolds so that they apply to discrete Schr\"odinger 
operators and higher order operators on the Euclidean space. The proof relies on 
propagation estimates, whereas the original proof of Isozaki-Kitada relies 
on a construction of parametrices. 
\end{abstract}

\section{Introduction}

In this Introduction, we present our main results for Schr\"odinger operators for simplicity. 
The results under more general settings are explained in Section~2. Let 
\[
H=-\frac12 \triangle +V(x) \quad \text{on } L^2(\re^d), \ d\geq 1, 
\]
be a Schr\"odinger operator with a potential $V\in C^\infty(\re^d)$, real-valued. 
We suppose there is $\mu>0$ such that for any multi-index $\a\in\ze_+^d$, 
\[
\bigabs{\pa_x^\a V(x)}\leq C_\a \jap{x}^{-\mu-|\a|}, \quad x\in\re^d, 
\]
with some $C_\a>0$, where $\jap{x}=(1+|x|^2)^{1/2}$. Then it is well-known that 
$\s_{\mathrm{\small ess}}(H)=[0,\infty)$; $H$ has no positive eigenvalues; 
$H$ is absolutely continuous on $(0,\infty)$, and 
\[
(H-\l\mp i0)^{-1} =\lim_{\e\to +0} (H-\l\mp i\e)^{-1}, \quad\l>0, 
\]
exist as operators from $L^2(\re^d,\jap{x}^s dx)$ to $L^2(\re^d,\jap{x}^{-s}dx)$ with 
$s>1/2$. We denote the Fourier transform by $\mathcal{F}$, and we write 
$\hat H =\mathcal{F} H \mathcal{F}^*$. Then the above claim implies 
$(\hat H-\l\mp i0)^{-1}$ exist as operators from $H^s(\re^d)$ to $H^{-s}(\re^d)$, 
and thus they have distribution kernel of order at most 1. 
We denote their distribution kernels by $K^\pm(\l)\in\mathcal{S}'(\re^{2d})$. 
We investigate the wave front set of $K^\pm(\l)$. 
We use somewhat nonstandard notation to represent a point in 
$T^*\re^{2d} \cong \re^{4d}$: We denote
\[
(x,\x,y,\y)\in T^*\re^{2d}, \quad\text{where } (\x,\y)\in \re^{2d}, \text{ and }
(x,y)\in T^*_{(\x,\y)}(\re^{2d}), 
\]
i.e., $\x,\y$ denote points in $\re^d$ (the Fourier space), and $x,y$ denote 
points in the cotangent spaces at $\x, \y$, respectively. We denote
\begin{align*}
\S_0 &=\bigset{(x,\x,-x,\x)}{(x,\x)\in T^*\re^d}, \\
\S_\pm(\l) &= \bigset{(x+t\x,\x,-x,\x)}{(x,\x)\in T^*\re^d, \tfrac12|\x|^2=\l, \pm t\geq 0}, \\
\S_\pm'(\l) &= \bigset{(t\x,\x)}{\tfrac12|\x|^2=\l, \pm t\geq 0}\times 
\bigset{(-t\x,\x)}{\tfrac12|\x|^2=\l, \mp t\geq 0}
\end{align*}
for $\l>0$. We denote the wavefront set of a distribution $T$ by $\WF(T)$. 

\begin{thm}\label{thm:1-1}
For $\l>0$, 
\[
\WF(K^\pm(\l))\subset \S_0\cup \S_\pm(\l)\cup \S_\pm'(\l). 
\]
\end{thm}

\begin{rem}
$\S_0$ denotes the diagonal set, and $\WF((\text{kernel of }A))\subset \S_0$ 
if $A$ is a pseudodifferential operator. $\S_\pm(\l)$ represent the {\em free propagation}\/
parts, and it is easy to show $\WF(K^\pm(\l))=\S_0\cup \S_\pm(\l)$ if $V=0$. 
Thus only the third part $\S_\pm'(\l)$ describes the singularities generated by the 
perturbation $V$. 
\end{rem}

\begin{rem}
A mocrolocal resolvent estimate of this form was proved in \cite{N3} 
for the short range case (in more general setting as in Section~2), and applied to 
the analysis of scattering matrices. The proof relies on a construction of parametrices. 
\end{rem}

The (two-sided) microlocal resolvent estimates of Isozaki-Kitada \cite{IK1,IK2} 
follow easily from Theorem~\ref{thm:1-1}. 

\begin{cor}\label{cor:1-2}
Let $a_\pm\in S^0_{1,0}(\re^d)$, i.e., $a_\pm\in C^\infty(\re^{2d})$ and 
for any multi-indices $\a,\b$, 
\[
\bigabs{\pa_x^\a\pa_\x^\b a_\pm(x,\x)}\leq C_{\a\b}\jap{x}^{-|\a|}, \quad x,\x\in\re^d
\]
with some $C_{\a\b}>0$. Suppose there are $0<c_1<c_2$ and  $-1<\c_-<\c_+<1$ such that 
\[
\supp[a_\pm] \subset \biggset{(x,\x)}{\pm\frac{x\cdot\x}{|x|\,|\x|} \geq \pm \c_\pm, 
c_1\leq |\x|\leq c_2, |x|\geq 1}. 
\]
Let $A_\pm =a_\pm(x, D_x)$. Then for any $N>0$, 
\[
\jap{x}^N A_\mp (H-\l\mp i0)^{-1} A_\pm^* \jap{x}^N \in B(L^2(\re^d)), \quad \l>0.
\]
\end{cor}

\begin{proof}
It suffices to show $\mathcal{F}A_\mp (H-\l\mp i0)^{-1} A_\pm^*\mathcal{F}^*$ are bounded from 
$H^{-N}(\re^d)$ to $H^N(\re^d)$, $\forall N>0$, i.e., they are smoothing operators. 
We note the distribution kernels of 
$\mathcal{F}A_\mp (H-\l\mp i0)^{-1} A_\pm^*\mathcal{F}^*$ are given by 
$a_\pm(-D_\x,\x) \overline{a_\mp}(D_\y,\y) K^\pm(\l;\x,\y)$. We also note that if 
$x\cdot \x \geq \c_+|x|\,|\x|$ then 
\[
(x+t\x)\cdot \x \geq \c_+|x|\,|\x|+t|\x|^2 \geq \c_+|x+t\x|\,|\x|, \quad t\geq 0, 
\]
and thus we have $\dfrac{x\cdot(x+t\x)}{|x|\,|x+t\x|}\geq \c_+$. This implies that if $(x,\x)\in\supp[a_+]$ 
then $(x+t\x,\x)\notin \supp[a_-]$, $t\geq 0$. Hence we learn that the essential support of 
$(a_-(x,\x) a_+(-y,\y))$ does not intersect $\S_+(\l)$. It is easy to show the essential support of 
$(a_-(x,\x) a_+(-y,\y))$ does not intersect $\S_0$ and $\S_+'(\l)$. These imply 
$\mathcal{F} A_-(H-\l-i0)^{-1} A_+^* \mathcal{F}^*$ is smoothing. 
Similarly, we can show $\mathcal{F} A_+(H-\l+i0)^{-1} A_-^* \mathcal{F}^*$ is smoothing, 
and we complete the proof.
\end{proof}

\begin{rem}
Corollary~\ref{cor:1-2} was proved by Isozaki and Kitada \cite{IK1,IK3}, and it is analogous to 
two sided resolvent estimates of Mourre \cite{M2} (see also G\'erard \cite{Ge}). 
The microlocal resolvent estimate of the above form is used to analyze 
long-range scattering and scattering matrices (\cite{IK2,IK4}). 
\end{rem}

In Section~2, we formulate our main results in more general settings. 
In Section~3, we prove our main theorem assuming a key lemma 
(Proposition~\ref{prop:3-1}), which is proved in Section~4. We discuss so-called one-sided 
microlocal resolvent estimates in Section~5. 

\noindent{\bf Acknowlegement:} A part of this work was done when the author was 
staying at Isaac Newton Institute for Mathematical Sciences for the program: 
{\it Periodic and Ergodic Spectral Problems},  supported by EPSRC Grant Number EP/K032208/1, 
and he thanks the institute and the Simons Foundation for the financial support and its hospitality. 

\section{Model and main theorem}

Here we formulate our model and state our main results that applies to higher order operators 
on $\re^d$ as well as various difference operators on $\ze^d$. 

Let $M$ be a $d$-dimensional $C^\infty$ Riemannian manifold with a smooth density $m$, 
and let $p_0(\x)$, $\x\in M$, be a real-valued smooth function on $M$. 
Let $\mu\in (0,1]$ and let $V\in S^{-\mu}_{1,0}(M)$, i.e., $V\in C^\infty(T^* M)$ and for any 
multi-indices $\a,\b\in \ze_+^d$ there is $C_{\a\b K}$ in each local coordinate
$K\Subset M$,  such that 
\[
\bigabs{\pa_x^\a\pa_\x^\b V(x,\x)} \leq C_{\a\b K}\jap{x}^{-\mu-|\a|}, \quad
\x\in K, x\in T^*_\x M,
\]
where the length of $x$ is defined by the Riemann metric on $T^*_\x M$. 
We denote the quantization of $V$ by $\hat V=V(-D_\x,\x)$, and we write 
$V=\hat V$ where there is no confusion. We denote $\mathcal{H}=L^2(M,m)$, and 
\[
H_0\f(\x) =p_0(\x)\f(\x) \quad \text{for }\f\in D(H_0)=\bigset{\f\in\mathcal{H}}{p_0\f\in\mathcal{H}}.
\]
It is easy to see $H_0$ is self-adjoint. We suppose $\hat V$ is symmetric, 
$H_0$-bounded, and
\[
H=H_0+V, \quad D(H)=D(H_0)
\]
is self-adjoint on $\mathcal{H}$. 

Let $I\Subset \re$ be an interval, and we consider $(H-\l\mp i0)^{-1}$ for $\l\in I$. 
We define the velocity by 
\[
v(\x)=dp_0(\x)\in T^*_\x M, \quad \x\in M. 
\]
We suppose $p_0^{-1}(I)=\bigset{\x\in M}{p_0(\x)\in I}$ is compact, and 
\[
v(\x)\neq 0\quad \text{for } \x\in p_0^{-1}(I),
\]
i.e., $I$ does not contain critical values of $p_0$. Under this assumption, it is easy to see 
that the following claims using the standard Mourre theory (see, e.g., \cite{M1}, \cite{ABG}, 
\cite{N3} Section~2): $\s_{\mathrm{\small p}}(H)\cap I$ is discrete, each eigenvalues are finite dimensional, and 
for $\l\in I\setminus \s_{\mathrm{\small p}}(H)$, $s>1/2$, the limits 
\[
(H-\l\mp i0)^{-1} =\lim_{\e\to+0} (H-\l\mp i\e)^{-1} \in B(H^s, H^{-s})
\]
exist. Let $K^\pm(\l)$ be the distribution kernels of $(H-\l\mp i0)^{-1}$, and 
we consider the microlocal singularities of $K^\pm(\l)$. As well as in the previous section, 
we represent a point in $T^*M$ by 
\[
(x,\x)\in T^*M, \ \text{where } \x\in M, x\in T^*_\x M, 
\]
and also $(x,\x,y,\y)\in T^*(M\times M)$, where $(\x,\y)\in M\times M$, $x\in T^*_\x M$ 
and $y\in T^*_\y M$. We set $\S_0$, $\S_\pm(\l)$, $\S_\pm'(\l)\subset T^*(M\times M)$ as 
\begin{align*}
\S_0&= \bigset{(x,\x,-x,\x)}{(x,\x)\in T^*M}, \\
\S_\pm(\l) &= \bigset{(x+tv(\x),\x, -x,\x)}{p_0(\x)=\l, \pm t\geq 0}, \\
\S_\pm'(\l) &= \bigset{(t v(\x),\x)}{p_0(\x)=\l, \pm t\geq 0}\times 
\bigset{(-t v(\x),\x)}{p_0(\x)=\l, \mp t\geq 0}.
\end{align*}
Then our main result is stated as follows: 

\begin{thm}\label{thm:2-1}
Let $\l\in I\setminus \s_{\mathrm{\small p}}(H)$. Then 
\[
\WF(K^\pm(\l))\subset \S_0\cup \S_\pm(\l)\cup \S_\pm'(\l).
\]
\end{thm}

Microlocal resolvent estimates of Isozaki-Kitada type follows from this analogously 
to the previous section. 

\begin{cor}\label{cor:2-2}
Let $\l\in I\setminus \s_{\mathrm{\small p}}(H)$, $a_\pm \in S^0_{1,0}(M)$, and suppose 
\[
\supp[a_\pm] \subset \biggset{(x,\x)\in T^*M}{\pm \frac{x\cdot v(\x)}{|x|\, |v(\x)|}\geq \pm\c_\pm, 
p_0(\x)\in K},
\]
where $-1<\c_-<\c_+<1$, $K\Subset M$. We set $A_\pm=a_\pm(-D_\x,\x)$. Then 
$A_\mp(H-\l\mp i0)^{-1}A_\pm^*$ are smoothing operators, bounded from $H^{-N}(M)$
to $H^N(M)$ with any $N$. 
\end{cor}

\noindent
{\bf Examples:} (1) A straightforward application is a differential operator on $\re^d$. 
Let $H_0$ be an $m$-th order symmetric elliptic partial differential operator with constant 
coefficients. We may write $H_0=p_0(D_x)$ with a real-valued polynomial of degree $m$. 
Suppose 
\[
V=\sum_{|\a|\leq m-1} b_\a(x) D_x^\a
\]
with $b_\x\in C^\infty(\re^d)$ for each $\a\in \ze_+^d$, $|\a|\leq m-1$. Suppose moreover that 
$V$ is symmetric and $\bigabs{\pa_\x^\b b_\a(x)}\leq C_{\a\b}\jap{x}^{-\mu-|\b|}$ 
for each $\a$ and $\b$. Let $I\Subset \re$ be an interval that does not contain critical 
points of $p_0(\x)$. Then Theorem~\ref{thm:2-1} applies for $\l\in I\setminus\s_{\mathrm{\small p}}(H)$. 

\noindent
(2) Another typical application is a difference operator on $\ze^d$. Let $H_0$ be a 
finite difference operator with constant coefficients, i.e., 
\[
H_0u(n) = \sum_{m\in K} \c_m u(n-m), \quad n\in\ze^d, 
\]
where $K\subset\ze^d$ is a finite subset, and $\c_m\in\co$, $m\in K$. We suppose 
$H_0$ is symmetric. Then 
\[
p_0(\x)= \sum_{m\in K} \c_m e^{i\x\cdot m}
\]
is a real-valued trigonometric polynomial on the torus $M=\mathbb{T}^d =(\re/2\pi\ze)^d$. 
Suppose $V(n)$ is the restriction of a smooth real-valued function $\tilde V(x)$ on $\re^d$ 
which satisfy $\bigabs{\pa_\x^\a V(x)}\leq C_{\a}\jap{x}^{-\mu-|\a|}$ for each $\a\in\ze_+^d$. 
Then we can apply Theorem~\ref{thm:2-1} to $H=H_0+V$. We refer Nakamura \cite{N3} Section~7 
for the detail of the construction. 

\section{Proof of Theorem~\ref{thm:2-1}}

Here we prove our main theorem assuming a proposition, which is proved in the next section. 

\subsection{Notation}
We use several classes of symbols. We denote the standard Kohn-Nirenberg symbol class 
of order $m$ by $S^{m}$, i.e., $a\in S^m$ if $a\in C^\infty(T^*M)$ and for any multi-indices 
$\a,\b\in\ze_+^d$, 
\[
\bigabs{\pa_x^\a\pa_\x^\b a(x,\x)}\leq C_{\a\b}\jap{x}^{m-|\a|}, \quad \x\in M, x\in T^*_\x M
\]
in each (relative compact) local coordinate with some $C_{\a\b}>0$. 
We often use $h$-dependent symbols. 
We denote $a(h,x,\x)\in S^m_h$ if $a(h,\cdot,\cdot)\in C^\infty(T^*M)$, $h\in (0,1]$, and 
for any $\a,\b\in\ze_+^d$, 
\[
\bigabs{\pa_x^\a\pa_\x^\b a(h,x,\x)}\leq C_{\a\b}\min\bigpare{\jap{x}^{m-|\a|}, h^{-m+|\a|}}
\]
for $\x\in M$, $x\in T^*_\x M$, $h\in (0,1]$ with some $C_{\a\b}>0$. 
For example, $a(hx,\x)\in S^0_h$ if $a(x,\x)\in C_0^\infty(T^*M)$ is supported away from $\{x=0\}$. 
Similarly, we use $(h,t)$-dependent symbols, usually supported in the region: $|x|=O(h^{-1}+t)$. 
We denote $a\in S^m_{h,t}$ if $a(h,t,\cdot,\cdot)\in C^\infty(T^*M)$, and for any $\a,\b\in\ze_+^d$, 
\[
\bigabs{\pa_x^\a\pa_\x^\b a(h,t,x,\x)}\leq C_{\a\b}\min\bigpare{\jap{x}^{m-|\a|}, (h^{-1}+t)^{m-|\a|}}
\]
for $\x\in M$, $x\in T^*_\x M$, $h\in (0,1]$, $t\geq 0$ with some $C_{\a\b}>0$. 

Our results are independent of the choice of quantizations, but we employ symmetric 
quantization for convenience. For a symbol $a$, we denote the symmetric quantization, 
e.g., the Weyl quantization $a^W(-D_\x,\x)$ by $\Op(a)$. We also denote the quantization 
of $a(h,hx,\x)$ by $\Op^h(a)$. We refer H\"ormander \cite{Ho} Vol.~3 for the pseudodifferential 
operator calculus. 

\subsection{Semiclassical reduction} 

We consider the ``$+$'' case only. The ``$-$'' case can be handled similarly. 
We suppose 
\[
(x_1,\x_1,-x_2,\x_2)\notin \S_0\cup \S_+(\l)\cup \S_+'(\l),
\]
where $\l\in I\setminus \s_{\mathrm{\small p}}(H)$, $(x_1,x_2)\neq 0$, and we show $(x_1,\x_1,-x_2,\x_2)\notin 
\WF(K^+(\l))$. By the well-known semiclassical characterization of the wave front set (see, e.g., 
Martinez \cite{Ma}), 
it suffice to show the existence of $a_0\in C^\infty(T^*(M\times M))$ such that 
\[
a_0(x_1,\x_1,-x_2,\x_2)\neq 0
\]
and
\[
\bignorm{a_0(-hD_\x,\x,-hD_\y,\y)K^+(\l;\x,\y)}_{L^2}\leq C_N h^N, \quad h\in (0,1],
\]
with any $N$. We consider the case 
\[
a_0(x,\x,-y,\y)= a_1(x,\x)a_2(y,\y),
\]
where $a_1,a_2\in C_0^\infty(T^*M)$ are real-valued. 
Then it is easy to see 
\begin{align*}
&\bignorm{a_0(-hD_\x,\x,-hD_\y,\y)K^+(\l,\x,\y)}_{L^2}\\
&\qquad= \bignorm{\Op^h(a_1)(H-\l-i0)^{-1}\Op^h(a_2)}_{HS} \\
&\qquad \leq C h^{-d}\bignorm{\Op^h(a_1)(H-\l-i0)^{-1}\Op^h(a_2)}_{B(L^2)},
\end{align*}
where $\norm{\cdot}_{HS}$ and $\norm{\cdot}_{B(L^2)}$ denote the Hilbert-Schmidt norm 
and the operator norm, respectively. Thus it suffices to find $a_1,a_2\in C_0^\infty(T^*M)$
such that $a_1(x_1,\x_1)\neq 0$, $a_2(x_2,\x_2)\neq 0$ and 
\begin{equation}\label{eq-3-1}
\bignorm{\Op^h(a_1)(H-\l-i0)^{-1}\Op^h(a_2)}_{B(L^2)} \leq C_N h^N, \quad h\in(0,1],
\end{equation}
for any $N$. In the following, we denote the operator norm of an operator $A$ by 
$\norm{A}$ without subscripts. 

\subsection{Case 1}
At first we consider the easy case, i.e., either $p_0(\x_1)\neq \l$ or $p_0(\x_2)\neq \l$. 
For the moment we suppose $p_0(\x_2)\neq \l$. Then we choose $a_2\in C_0^\infty(T^*M)$ 
such that $a_2(x_2,\x_2)=1$ and 
\[
\supp[a_2]\subset \bigset{(x,\x)}{|p_0(\x)-\l|>3\e}
\]
with some $\e>0$. We then choose $f\in C_0^\infty(\re)$ such that $f(z)=1$ on 
$(\l-\e,\l+\e)$ and $\supp[f]\subset [\l-2\e,\l+2\e]$. By the functional calculus, 
$f(H)$ is a pseudodifferential operator with the symbol in $S^0$, and the 
symbol is supported in $p_0^{-1}([\l-2\e,\l+2\e])$ modulo $S^{-\infty}=\bigcap_N S^{-N}$
(see, e.g., Dimassi-Sj\"ostrand \cite{DS}). Hence, by the asymptotic 
expansion, we learn $f(H) \Op^h(a_2)$ has a symbol in $S_h^{-\infty} =\bigcap_N S^{-N}_h$. 
In particular, we have
\[
\bignorm{\jap{D_\x} f(H)\Op^h(a_2)}\leq C_N h^N, \quad h\in (0,1], 
\]
with any $N$. On the other hand, noting $(z-1)^{-1}(1-f(z))\in S^0(\re)$, we learn 
$(H-\l-i0)^{-1}(1-f(H))$ is a pseudodifferential operator with the symbol in $S^0$. 
We may suppose $\supp[a_1]\cap
\supp[a_2]=\emptyset$, and hence
\begin{equation}\label{eq-3-2}
\bignorm{\Op^h(a_1)(H-\l-i0)^{-1}(1-f(H))\Op^h(a_2)}\leq C_N h^N
\end{equation}
with any $N$. Combining these, we have 
\begin{align*}
&\bignorm{\Op^h(a_1)(H-\l-i0)^{-1}\Op^h(a_2)}\\
&\leq \bignorm{\Op^h(a_1)\jap{D_\x}}\, \bignorm{\jap{D_\x}^{-1}(H-\l-i0)^{-1}\jap{D_\x}^{-1}}\,
\bignorm{\jap{D_\x}f(H) \Op^h(a_2)} \\
&\quad+\bignorm{\Op^h(a_1)(H-\l-i0)^{-1}(1-f(H))\Op^h(a_2)}\\
&\leq C_N' h^{N-2}, \quad h\in (0,1], 
\end{align*}
since $\bignorm{\Op^h(a_j)\jap{D_\x}} =O(h^{-1})$ as $h\to+0$. 
This proves \eqref{eq-3-1}. The case $p_0(\x_1)\neq \l$ is handled similarly. 

\subsection{Case 2}
We now suppose $p_0(\x_1)=p_0(\x_2)=\l$. We choose $f\in C_0^\infty(\re)$ so that 
$\supp[f]\Subset (I\setminus \s_{\mathrm{\small p}}(H))$ and $f=1$ on $[\l-\e,\l+\e]$ with some $\e>0$. 
Since $(H-\l-i0)^{-1}(1-f(H))$ is a pseudodifferential operator, \eqref{eq-3-2} holds as well. 
Thus it suffices to consider $\Op^h(a_1)(H-\l-i0)^{-1}f(H)\Op^h(a_2)$. We recall 
\[
(H-\l-i0)^{-1} =i\lim_{\e\to+0} \int_0^\infty e^{-it(H-\l-i\e)} dt 
=i\int_0^\infty e^{it\l} e^{-itH} dt
\]
in the weak sense. Thus it suffices to show 
\begin{equation}\label{eq-3-3}
\int_0^\infty \bignorm{\Op^h(a_1)e^{-itH}f(H)\Op^h(a_2)}dt 
\leq C_N h^N, \quad h\in (0,1],
\end{equation}
for any $N$. 

\begin{prop}\label{prop:3-1}
Let $(x_1,\x_1,-x_2,\x_2)\notin \S_0\cup\S_+(\l)\cup \S_+'(\l)$, and 
$p_0(\x_1)=p_0(\x_2)=\l$. If $a_j$ are supported in sufficiently small 
neighborhoods of $(x_j,\x_j)$, $j=1,2$, then for any $N$ there is $C_N$ such that 
\[
\bignorm{\Op^h(a_1)e^{-itH} f(H) \Op^h(a_2)} \leq C_N h^N, \quad h\in (0,1], t\geq 0. 
\]
\end{prop}

\begin{rem}
Here we do not assume $\l\notin \s_{\mathrm{\small p}}(H)$. Thus the integrability in $t$ does not 
necessarily hold. We also note that we assume $(x_1,x_2)\neq 0$, but one of 
$\{x_1,x_2\}$ may be 0. 
\end{rem}

We prove Proposition~\ref{prop:3-1} in the next section, and we complete the proof of Theorem~\ref{thm:2-1}
assuming Proposition~\ref{prop:3-1}. By the multiple commutator estimate (Jensen-Mourre-Perry
\cite{JMP}), we have the following standard local decay estimate: for any $\nu>\k>0$, 
there is $C$ such that 
\begin{equation}\label{eq-3-4}
\bignorm{\jap{D_\x}^{-\nu} e^{-itH} f(H) \jap{D_\x}^{-\nu}}\leq C\jap{t}^{-\k}, 
\quad t\in\re,
\end{equation}
provided $f$ is supported in $I\setminus \s_{\mathrm{\small p}}(H)$. We choose $\k=2$, $\nu=3$, and then 
we have 
\begin{align*}
&\bignorm{\Op^h(a_1)e^{-itH} f(H) \Op^h(a_2)}\\
&\quad \leq \bignorm{\Op^h(a_1)\jap{D_\x}^3}\, \bignorm{\jap{D_\x}^{-3} e^{-itH} f(H) \jap{D_\x}^{-3}}\,
\bignorm{\jap{D_\x}^3 \Op^h(a_2)} \\
&\quad \leq C h^{-6}\jap{t}^{-2}, \quad h\in(0,1],t\in\re,
\end{align*}
where we have used $\bignorm{\Op^h(a_j)\jap{D_\x}^3} =O(h^{-3})$. 
For an arbitrary $M>0$, we set $N=2M+6$ in Proposition~\ref{prop:3-1}, and $T=h^{-M-6}$. 
Then we learn 
\begin{align*}
\int_0^\infty \bignorm{\Op^h(a_1)e^{-itH} f(H) \Op^h(a_2)}dt 
&\leq \int_0^T\cdots dt +\int_T^\infty \cdots dt \\
& \leq C_N h^{2M+6} h^{-M-6} +C h^{-6} h^{M+6} \\
&\leq CM' h^{-M}, \quad h\in (0,1].
\end{align*}
This implies \eqref{eq-3-3}, and hence Theorem~\ref{thm:2-1}. \qed 

\section{Propagation estimate : Proof of Proposition~\ref{prop:3-1}}

We employ propagation estimate argument similar to that in Nakamura \cite{N1}. 
We note that the Egorov-type argument works for each $t$, but not uniformly in $t>0$. 
Thus we cannot apply the Egorov-type argument directly here. 

Let $(x_1,\x_1,-x_2,\x_2)$ as in the proposition. Since $(x_1,\x_1,-x_2,\x_2)\notin \S_+'(\l)$, 
either $x_1+t v(\x_1)\neq 0$ for $t\leq 0$, or $x_2+tv(\x_2)\neq 0$ for $t\geq 0$. 
We first consider the later case, i.e., 
\[
x_2+tv(\x_2)\neq 0, \quad t\geq 0.
\]
We remark that this assumption implies $x_2\neq 0$, but the case $x_1=0$ is not
excluded. 
Then there exist $\d_1,\d_2>0$ such that 
\begin{equation}\label{eq-4-1}
\O(t)\cap \{(x_1,\x_1)\}=\emptyset, \quad \O(t)\cap(\{0\}\times M)=\emptyset, 
\quad\text{for } t\geq 0, 
\end{equation}
where 
\[
\O(t)=\bigset{(x,\x)}{|x-(x_2+tv(\x_2))|\leq 3\d_1(1+t), |\x-\x_2|\leq \d_2}.
\]
We may also suppose $\d_2$ is so small that: 
\begin{equation}\label{eq-4-2}
\text{if }|\x-\x_2|\leq 2\d_2 \text{ then } |v(\x)-v(\x_2)|<\d_1/2.
\end{equation}
We choose $\F\in C^\infty([0,\infty))$ such that $\F(s)=1$ if $s\leq 1/2$; $\F(s)=0$ 
if $s\geq 1$; $\F(s)>0$ if $s<1$; and $\F'(s)\leq 0$ for $s\leq 1$. We also write 
$\G(s)=\F(s)^2$, $s\geq 0$. We now set
\[
a_j(x,\x) =\F\biggpare{\frac{|x-x_j|}{\d_1}}\F\biggpare{\frac{|\x-\x_j|}{\d_2}}, 
\quad (x,\x)\in T^*M, j=1,2, 
\]
then $a_j\in S^0$, and $a_j(hx,\x)\in S^0_h$. 
We also set
\[
\phi_0(t,x,\x)= \F\biggpare{\frac{|x-y(t)|}{\d_1(h^{-1}+t)}}\F\biggpare{\frac{|\x-\x_2|}{\d_2}}, 
\quad (x,\x)\in T^*M, t\geq 0, 
\]
where 
\[
y(t)=h^{-1}x_2+tv(\x_2) = h^{-1}(x_2+htv(\x_2)).
\]
We note, by the condition \eqref{eq-4-1}, 
\[
|y(t)|\geq 3\d_1h^{-1}(1+ht), \quad t\geq 0. 
\]
On the other hand, by the support property of $\F$, we have 
\[
|x-y(t)|\leq \d_1 h^{-1} (1+ht)
\]
on the support of $\phi_0(t;\cdot,\cdot)$. Hence we learn 
\[
2\d_1h^{-1}(1+ht) \leq |x|\leq C h^{-1}(1+ht)
\]
with some $C$ on the support of $\phi_0(t;\cdot,\cdot)$, and 
this implies $\phi_0 \in S^0_{h,t}$. We denote the support of $\phi_0(t,\cdot,\cdot)$ by 
\[
\O_0(t)=
\bigset{(x,\x)}{|x-y(t)|\leq \d_1(h^{-1}+t), |\x-\x_2|\leq \d_2}.
\]
Now let $\g_0(t,x,\x)$ be the symbol of $|\Op(\phi_0(t,\cdot,\cdot))|^2$. 
Clearly $\g_0\in S^0_{h,t}$ and the principal symbol is 
\[
\g_0^0(t,x,\x)= \G\biggpare{\frac{|x-y(t)|}{\d_1(h^{-1}+t)}}\G\biggpare{\frac{|\x-\x_2|}{\d_2}}, 
\]
i.e., $\g_0-\g_0^0\in S^{-1}_{h,t}$. We note $\g_0$ is supported in $\O_0(t)$ modulo 
$S^{-\infty}_{h,t}$. 

Then we compute 
\begin{align*}
&\pa_t \g_0^0(t,x,\x) +v(\x)\cdot \pa_x\g_0^0(t,x,\x) \\
&= \frac{1}{\d_1(h^{-1}+t)}\biggbra{-\frac{|x-y(t)|}{h^{-1}+t}
+\frac{x-y(t)}{|x-y(t)|}\cdot(v(\x)-v(\x_2))}\times \\
&\qquad\times \G'\biggpare{\frac{|x-y(t)|}{\d_1(h^{-1}+t)}}\G\biggpare{\frac{|\x-\x_2|}{\d_2}}.
\end{align*}
Since 
\[
\frac{\d_1}{2}(h^{-1}+t)\leq |x-y(t)|, \quad |v(\x)-v(\x_2)|\leq \frac{\d_1}{2}
\]
on the support, we learn $\{\cdots\}$ in the RHS is nonpositive. Recalling $\G'(s)\leq 0$, 
we learn 
\begin{equation}\label{eq-4-3}
\pa_t\g_0^0(t,x,\x)+v(\x)\cdot\pa_x \g_0^0(t,x,\x)\geq 0, \quad (x,t)\in T^*M, t\geq 0.
\end{equation}
We also note $\pa_t\g^0_0, \pa_x\g^0_0\in S^{-1}_{h,t}$. Then by the sharp G{\aa}rding 
inequality and asymptotic expansions, we learn 
\[
\pa_t\Op(\g_0^0)+i[H_0,\Op(\g_0^0)]\geq \Op(r_0^0)
\]
with some $r_0^0\in S^{-2}_{h,t}$. We then have, using the assumption on $V$, 
\[
\pa_t\Op(\g_0)+i[H,\Op(\g_0)]\geq \Op(r_0)
\]
with some $r_0\in S^{-1-\mu}_{h,t}$, supported in $\O(t)$ modulo $S^{-\infty}_{h,t}$. 

Now we choose constants $\c_j$, $j=1,2,\dots$, so that $1<\c_1<\c_2<\cdots<2$. 
Let $C_j>0$, $j=1,2,\dots$, be constants decided later. We then set
\[
\g_j(t,x,\x)= C_jh^{(j-1)\mu}\bigpare{h^\mu-(h^{-1}+t)^{-\mu}}
\G\biggpare{\frac{|x-y(t)|}{\c_j\d_1(h^{-1}+t)}}\G\biggpare{\frac{|\x-\x_2|}{\c_j\d_2}}
\]
for $(x,\x)\in T^*M$, $t\geq 0$ and $j=1,2,\dots$. 
By direct computations, we see $\g_j\in h^{j\mu} S^0_{h,t}$ and $\pa_t\g_j\in h^{j\mu}S^{-1}_{h,t}$. 
Moreover, we have 
\[
\pa_t\g_j+v(\x)\cdot\pa_x\g_j \geq \mu C_jh^{(j-1)\mu}(h^{-1}+t)^{-1-\mu}
\G\biggpare{\frac{|x-y(t)|}{\c_j\d_1(h^{-1}+t)}}\G\biggpare{\frac{|\x-\x_2|}{\c_j\d_2}},
\]
which is proved similarly to \eqref{eq-4-3}. We set 
\[
\O_j(t)=\bigset{(x,\x)}{|x-y(t)|\leq \c_j\d_1(1+t), |\x-\x_2|\leq \c_j\d_2}.
\]
Then $\g_j(t,x,\x)$ are supported in $\O_j(t)$, and 
\[
\pa_t\g_j + v(\x)\cdot\pa_x\g_j(t,x,\x) \geq \mu\k_j C_j h^{(j-1)\mu} (h^{-1}+t)^{-1-\mu}
\quad \text{on $\O_{j-1}(t)$},
\]
$j=1,2,\dots$, where $\k_j>0$ are constants depending only on 
the choice of $\{\c_j\}$ and $\G$. Hence, if we choose $C_1$ sufficiently large, 
we have 
\[
\pa_t\g_1+v(\x)\cdot\pa_x\g_1 +r_0\geq 0 \quad \text{on }T^*M\times([0,\infty).
\]
Then by using the sharp G{\aa}rding inequality again, we have 
\[
\pa_t\Op(\g_0+\g_1)+i[H,\Op(\g_1+\g_2)]\geq \Op(r_1)
\]
with some $r_1\in h^\mu S^{-1-\mu}_{h,t}$, supported in $\O_1(t)$ modulo $S^{-\infty}_{h,t}$. 
Repeating this procedure, we decide $C_2,C_3,\dots$, and we have 
\[
\pa_t\Bigpare{\Op\Bigpare{{\textstyle \sum_{j=1}^m\g_j}}}
+i\Bigbrac{H,\Op\Bigpare{{\textstyle \sum_{j=1}^m\g_j}}}\geq \Op(r_m),
\]
where $r_m \in h^{m\mu}S^{-1-\mu}_{h,t}$, supported in $\O_m(t)$ modulo $S^{-\infty}_{h,t}$. 
In particular, we have 
\[
\int_0^\infty \bignorm{\Op(r_m)}dt \leq Ch^{m\mu}\int_0^\infty (h^{-1}+t)^{-1-\mu}dt
\leq C' h^{(m+1)\mu}.
\]
We fix $m$ large enough so that $(m+1)\mu>2N$, where $N$ is the exponent in Proposition~3.1. 

Then we set
\[
\g(t,x,\x)=\sum_{j=1}^m \g_j(t,x,\x)\in S^0_{h,t}.
\]
We summarize the properties of $\g$. 

\begin{lem}\label{lem:4-1}
$\g$ and $F(t)=\Op(\g(t,\cdot,\cdot)$ satisfy the following properties:
\begin{enumerate}
\renewcommand{\labelenumi}{{\rm (\theenumi)}}
\item $\g\in S^0_{h,t}$ and $F(0)=|\Op^h(a_2)|^2$. 
\item $g$ is supported in 
\[
\tilde \O(t)= \bigset{(x,\x)}{|x-y(t)|\leq 2\d_1(h^{-1}+t),|\x-\x_2|\leq 2\d_2}
\]
modulo $S^{-\infty}_{h,t}$.
\item $F(t)$ satisfies the energy inequality:
\[
\pa_t F(t)+i[H,F(t)]\geq R(t),
\]
where $\int_0^\infty \norm{R(t)}dt \leq Ch^{2N}$. 
\end{enumerate}
\end{lem}

\begin{proof}[Proof of Proposition~\ref{prop:3-1}]
We recall the Heisenberg equation:
\[
\frac{d}{dt}\bigpare{e^{itH}F(t)e^{-itH}} =e^{itH} \bigpare{\pa_t F(t)+i[H,F(t)]}e^{-itH},
\]
and hence we have 
\[
\frac{d}{dt}\bigpare{e^{itH} F(t) e^{-itH}} \geq e^{itH} R(t) e^{-itH}.
\]
Integrating this inequality, we learn 
\[
e^{itH}F(t)e^{-itH} -F(0) \geq \int_0^t e^{itH} R(t) e^{-itH} dt \geq -C h^{2N}
\]
for all $t\geq 0$ by Lemma~\ref{lem:4-1}(3). Then, by using Lemma~\ref{lem:4-1}(1), we have 
\[
e^{-itH} |\Op^h(a_2)|^2 e^{itH} \leq F(t)+Ch^{2N},
\]
and hence 
\[
\bigabs{\Op^h(a_1) e^{-itH}\Op^h(a_2)}^2 \leq \Op^h(a_1)F(t)\Op^h(a_1)+C h^{2N}.
\]
We recall $\supp[a_1(h\cdot,\x)]\cap \supp[\g(t,\cdot,\cdot)]=\emptyset$;  
$\g\in S^0_{h,t}$ and hence $\g(t,\cdot,\cdot)$ is uniformly bounded in $S^0_h$. 
Then, by the asymptotic expansion, we learn $\norm{\Op^h(a_1)F(t)}=O(h^{2N})$, $h\to +0$, 
uniformly in $t\geq 0$. These imply 
\[
\bignorm{\Op^h(a_1) e^{-itH}\Op^h(a_2)}^2\leq Ch^{2N},
\]
and we complete the proof of Proposition~\ref{prop:3-1}, provided $x_2+tv(\x_2)\neq 0$ for $t\geq 0$. 

We now turn to the case $x_1+tv(\x_1)\neq 0$ for $t\leq 0$. We consider 
\[
(\Op^h(a_1) e^{-itH} \Op^h(a_2))^*= \Op^h(a_2)e^{itH} \Op^h(a_1),
\]
and replace $t$ by $-t$. Then is is easy to check $(x_2,\x_2,-x_1,\x_1)$ satisfies 
the conditions in the other case. Thus the conclusion follows from the same argument as above. 
\end{proof}

\section{One-sided estimates}

In Isozaki-Kitada \cite{IK1,IK3}, another kind of estimates, called one-sided microlocal 
resolvent estimates, are proved. In this section, we formulate the one-sided estimates 
under our setting, and we show they are proved by the same method used to prove 
Theorem~\ref{thm:2-1}. 

\begin{thm}\label{thm:5-1}
Let $\l\in I\setminus \s_{\mathrm{\small p}}(H)$ and suppose $a_\pm\in S^0(M)$ such that 
\[
\supp[a_\pm]\subset \biggset{(x,\x)\in T^*M}{\pm \frac{x\cdot v(\x)}{|x|\,|v(\x)|}>
\pm(-1+\e), p_0(\x)\in K}
\]
where $\e>0$, $K\Subset M$. Let $\nu>1$, $0<s<\nu-1$. 
Then $(H-\l\mp i0)^{-1}\Op(a_\pm)$ are bounded from $H^{-s}(M)$ 
to $H^{-\nu}(M)$. 
\end{thm}

We consider the ``+'' case only. The other case is proved similarly. 
Suppose $(x_2,\x_2)\in \supp[a_+]$. Then $x_2+tv(\x_2)\neq 0$ for 
$t\geq 0$, and we can construct the symbols used in Section~4. 
We use the same notation as in Section~4, and we use the same time-dependent 
symbol $\g(t,x,\x)$ constructed from $a_2(x,\x)$, which is supported in a small neighborhood 
of $(x_2,\x_2)$. Let $f \in C_0^\infty(\re)$ also as in Section~4. 
We choose $\chi\in C_0^\infty(M)$ such that 
$\chi(\x)f(H)=f(H)$ modulo $S^{-\infty}$. We then set
\[
\z(t,x,\x) =\G\biggpare{\frac{2|x|}{\d_1(h^{-1}+t)}}\chi(\x), \quad (x,\x)\in T^*M, t\geq 0, 
\]
and $\bar\z(t,x,\x)=\chi(\x)-\z(t,x,\x)=(1-\G(\cdots))\chi(\x)$. 
We note $\z,\bar\z\in S^0_{h,t}$. Then we observe $\norm{\Op(\z)\Op(\g)}=O(h^{2N})$ as $h\to+0$, 
uniformly in $t\geq 0$, again as in Section~4, since $\supp[\z]\cap \supp[\g]=\emptyset$ 
modulo $S^{-\infty}_{h,t}$. Thus we arrive at the following estimate, analogously to 
Proposition~\ref{prop:3-1}:

\begin{lem}
For any $N$, there is $C_N>0$ such that
\[
\bignorm{\Op(\z(t,\cdot,\cdot))e^{-itH} f(H) \Op^h(a_2)}\leq C_N h^N, \quad h\in (0,1], t\geq 0.
\]
\end{lem}

We then have, using the decomposition $1=\z+\bar\z+(1-\chi)$, 
\begin{align*}
&\bignorm{\jap{D_\x}^{-\nu} e^{-itH} f(H) \Op^h(a_2)} \\
&\leq \bignorm{\jap{D_\x}^{-\nu}\Op(\bar\z)}\, \bignorm{e^{-itH} f(H) \Op^h(a_2)} \\
&\quad + \bignorm{\jap{D_\x}^{-\nu}}\, \bignorm{\Op(\z)e^{-itH}f(H) \Op^h(a_2)} \\
&\quad + \bignorm{\jap{D_\x}^{-\nu}}\, \bignorm{(1-\chi(\x))f(H)}\, \bignorm{e^{-itH}f(H) \Op^h(a_2)}\\
&\leq C(h^{-1}+t)^{-\nu} +C_N h^N
\end{align*}
for $h\in (0,1]$, $t\geq 0$. On the other hand, by the local decay estimate \eqref{eq-3-4}, 
we have 
\[
\bignorm{\jap{D_\x}^{-\nu}e^{-itH} f(H) \Op^h(a_2)} \leq C h^{-\nu}\jap{t}^{-\k},
\quad t\in\re,
\]
with $1<\k<\nu$. By setting $T=h^{-N/2}$ and choosing $N$ large enough, we have 
\begin{align*}
&\int_0^\infty \bignorm{\jap{D_\x}^{-\nu} e^{-itH} f(H) \Op^h(a_2)}dt 
\leq \int_0^T\cdots dt +\int_T^\infty \cdots dt \\
&\quad \leq C\int_0^\infty (h^{-1}+t)^{-\nu}dt + C_N h^{N/2} + C h^{-\nu+(\k-1)N/2}
\leq C h^{\nu-1}.
\end{align*}
Thus we obtain the following:

\begin{lem}
Let $\nu>1$. Then 
\[
\bignorm{\jap{D_\x}^{-\nu} (H-\l-i0)^{-1} \Op^h(a_2)}\leq C h^{\nu-1}, \quad h\in (0,1]. 
\]
\end{lem}

Now suppose $\tilde a\in S^0(M)$ such that its essential support in contained in a 
small conic neighborhood of $(x_2,\x_2)$. Then by the standard Littlewood-Paley 
decomposition argument, we learn 
\[
(H-\l-i0)^{-1}f(H)\Op(\tilde a)\quad \text{is bounded from $H^{-s}(M)$ to 
$H^{-\nu}(M)$},
\]
where $0<s<\nu-1$. Since $(H-\l-i0)^{-1}(1-f(H))$ is a pseudodifferential operator, 
it is also bounded from $H^{-s}(M)$ to $H^{-s}(M)\subset H^{-\nu}(M)$. Thus 
$(H-\l-i0)^{-1}\Op(\tilde a)$ is bounded from $H^{-s}(M)$ to $H^{-\nu}(M)$. 
Now Theorem~\ref{thm:5-1} follows by the partition of unity argument. \qed

%
\vskip24pt
\small
%
\centerline{\bf References}
\vskip12pt

\begin{enumerate}
\renewcommand{\labelenumi}{[\arabic{enumi}]}
\renewcommand{\makelabel}{\rm}
\setcounter{enumi}{0}
\setlength{\itemsep}{-3pt}
\setlength{\parsep}{0cm}

\bibitem{ABG} Amrein, W., Boutet de Monvel, A.,  Georgescu, V.: 
$C_0$-groups, commutator methods and spectral theory of $N$-body Hamiltonians.
Progress in Mathematics, 135. Birkh\"auser Verlag, Basel, 1996. 

\bibitem{DS} Dimassi, M., Sj\"ostrand:
{\it Spectral Asymptotics in the Semi-Classical Limit}\/  
(London Mathematical Society Lecture Note Series), 
Cambridge University Press, 1999. 

\bibitem{Ge}  G\'erard, C.: 
A proof of the abstract limiting absorption principle by energy estimates. 
J. Funct. Anal. {\bf 254} (2008), no. 11, 2707--2724.

\bibitem{Ho} H\"ormander, L.: 
The Analysis of Linear Partial Differential Operators.  I--IV, Springer-Verlag, New York, 1983--1985. 

\bibitem{IK1} Isozaki, H., Kitada, H.: Microlocal resolvent estimates for 2-body Schr\"odinger operators. 
J. Funct. Anal. {\bf 57} (1984), no. 3, 270--300.

\bibitem{IK2} Isozaki, H., Kitada, H.: Modified wave operators with time-independent modifiers. 
J. Fac. Sci. Univ. Tokyo Sect. IA Math. {\bf 32} (1985), no. 1, 77--104.

\bibitem{IK3} Isozaki, H., Kitada, H.: A remark on the microlocal resolvent estimates for two body 
Schr\"odinger operators. 
Publ. Res. Inst. Math. Sci. {\bf 21} (1985), no. 5, 889--910.

\bibitem{IK4} Isozaki, H., Kitada, H.: Scattering matrices for two-body Schr\"odinger operators. 
Sci. Papers College Arts Sci. Univ. Tokyo {\bf 35} (1985), no. 2, 81--107.

\bibitem{JMP} Jensen, A., Mourre, E., Perry, P.:
Multiple commutator estimates and resolvent smoothness in quantum scattering theory. 
Ann. Inst. H. Poincar\'e Phys. Th\'eor. {\bf 41} (1984), no. 2, 207--225. 

\bibitem{Ma} Martinez, A.: 
{\it An Introduction to Semiclassical and Microlocal Analysis}, Universitext, 
Springer Verlag, 2002. 

\bibitem{M1} Mourre, E.: 
Absence of singular continuous spectrum for certain selfadjoint operators.
Comm. Math. Phys. {\bf 78} (1980/81), no. 3, 391--408. 

\bibitem{M2} Mourre, E.: Operateurs conjugu\'es et propri\'et\'es de propagation. 
Comm. Math. Phys. {\bf 91} (1983), no. 2, 279--300.

\bibitem{N1}Nakamura, S.:
Propagation of the homogeneous wave front set for Schr\"odinger equations. 
Duke Math. J. {\bf 126}  (2005), 349--367.

\bibitem{N2} Nakamura, S.:
Modified wave operators for discrete Schr\"odinger operators with long-range perturbations.
J. Math. Phys. {\bf 55} (2014), 112101 (8 pages).

\bibitem{N3} Nakamura, S.:
Microlocal properties of scattering matrices. To appear in Comm. P. D. E. 
(Preprint: {\tt http://arxiv.org/abs/1407.8299})

\end{enumerate}

\end{document}